\documentclass[reqno]{amsart}
\usepackage{amssymb,amsmath,amsthm,amscd,latexsym,amsfonts}
\usepackage{mathtools}
\usepackage[english]{babel}
\usepackage[T1]{fontenc}
\usepackage{inputenc}
\usepackage[arrow,matrix]{xy}
\usepackage{amsaddr}
\usepackage{graphicx}
\usepackage{cite}
\newtheorem{thm}{Theorem}[section]

\newtheorem{lem}[thm]{Lemma}
\newtheorem{prop}[thm]{Proposition}

\theoremstyle{definition}
\newtheorem{defn}[thm]{Definition}

\theoremstyle{remark}

\numberwithin{equation}{section}

\begin{document}

\title[On Leibniz superalgebras which even part is $\mathfrak{sl}_2$]{On Leibniz superalgebras which even part is $\mathfrak{sl}_2$}

\author{Kh.A.~Khalkulova\textsuperscript{1}, A.Kh.~Khudoyberdiyev\textsuperscript{1,2}}

\address{\textsuperscript{1} Institute of Mathematics Academy of Science of Uzbekistan, 81, Mirzo Ulug'bek street, 100125, Tashkent, Uzbekistan, xalkulova@gmail.com}

\address{\textsuperscript{2} National University of Uzbekistan, 4, University street, 100174, Tashkent,  Uzbekistan, khabror@mail.ru}

\begin{abstract} This article deals with a Leibniz superalgebra $L=L_0\oplus L_1,$ whose even part is a simple Lie algebra $\mathfrak{sl}_2$.
We describe all such Leibniz superalgebras when odd part is an irreducible Leibniz bi-module on $\mathfrak{sl}_2 $.
We show that there exist such Leibniz superalgebra with nontrivial odd part only in case of $dim L_1=2.$
\end{abstract}

\maketitle






\section{Introduction}

Extensive investigations in Lie algebras theory have led to the appearance of more general
algebraic objects -- Mal'cev algebras, binary Lie algebras, Lie superalgebras, Leibniz algebras and others.

During many years the theory of Lie superalgebras has been actively studied
by many mathematicians and physicists. A systematic exposition of basic of Lie
superalgebras theory can be found in \cite{Kac}. Many works have been devoted to the
study of this topic, but unfortunately most of them do not deal with nilpotent
Lie superalgebras. In works \cite{GL},
\cite{G-K-N} the problem of the description of some
classes of nilpotent Lie superalgebras have been studied.

Leibniz algebras have been first
introduced by Loday in \cite{Loday} as a non-antisymmetric version of Lie algebras. Leibniz superalgebras are generalizations of Leibniz algebras and, on the other hand, they naturally generalize Lie superalgebras. In the description of Leibniz superalgebras structure
the crucial task is to prove the existence of suitable basis (the so-called adapted basis) in which
the multiplication of the superalgebra has the most convenient form.

In the work \cite{G-K-N} the Lie superalgebras with maximal nilindex were classified. Such
superalgebras are two-generated and its nilindex equal to $n + m$ (where $n$ and $m$
are dimensions of even and odd parts, respectively). In fact, there exists unique Lie
superalgebra of maximal nilindex. This superalgebra is filiform Lie superalgebra
(the characteristic sequence equal to $(n-1, 1 | m)$) and we mention about paper \cite{GL},
where some crucial properties of filiform Lie superalgebras are given.

For nilpotent Leibniz superalgebras it turns to be comparatively
easy and was solved in \cite{Alb}. The distinctive property of such Leibniz superalgebras is that
they are single-generated and have the nilindex $n+m+1$.
The next step -- the description of Leibniz superalgebras with
dimensions of even and odd parts, respectively equal to $n$ and $m$, and with nilindex $n + m$ were classified
by applying restrictions the invariant such called characteristic sequences in \cite{FilSup}, \cite{C-G-N-O}, \cite{C-G-O-Kh}, \cite{G-O-Kh}. Solvable and semi-simple Leibniz superalgerbas are not investigated at this time. The first step of this assertion is describe Leibniz superalgebras which even part is a semi-simple Lie algebra. Note that the odd part of the superalgebra can be considered as a representation of the even part. Representation or bimodule of a Leibniz algebra $L$ is defined in \cite{LP_Universal} as a $\mathbb{K}$-module $M$ with two actions -- left and right, satisfying compatibility conditions. In \cite{Barnes} it is established that any simple finite-dimensional Leibniz representation is either symmetric, meaning the left and the right actions differ by sign, or antisymmetric, meaning  the left action is zero.  The classical Weyl's theorem on complete reducibility that claims any finite-dimensional module over a semisimple Lie algebra is a direct sum of simple modules does not generalize even for the simple Leibniz algebras case.

A Lie algebra can be considered as a Leibniz algebra and one can consider Leibniz representation of a Lie algebra. In \cite{LP_Leib_rep} the authors describe the indecomposable objects of the category of Leibniz representations of a Lie algebra and as an example, in case the Lie algebra is $\mathfrak{sl}_2$ the indecomposable objects in that category can be described, whereas for $\mathfrak{sl}_n$ ($n>2$) they claim that it is of wild type.

Our main focus in this work is to describe Leibniz superalgebras even part is isomorphic to the three dimensional simple Lie algebra $\mathfrak{sl}_2.$ If the multiplication of the odd part is zero, then Leibniz superalgebra is isomorpic to Leibniz algebra, i.e., superalgebra with trivial odd part. We show that there exist such Leibniz superalgebra $L=\mathfrak{sl}_2\oplus L_1$ with nontrivial odd part only in case of $dim L_1=2.$

Throughout this work we shall consider spaces, algebras and superalgebras over the field
of complex numbers.

\section{Preliminaries}

In this section we give necessary definitions and preliminary results.

\begin{defn}
    An algebra $(L,[\cdot,\cdot])$ over a field $\mathbb{K}$ is called a \emph{Leibniz algebra} if it is defined by the Leibniz identity
    $$[x,[y,z]]=[[x,y],z] - [[x,z],y], \ \mbox{for all}\ x,y,z \in L.$$
\end{defn}

In fact for Leibniz algebra $L$ the ideal
$I= span \{ [x,x] \ | \  x\in L \}$ coincides with the space spanned
by squares of elements of $L.$ Moreover, it is readily to see that
this ideal belongs to right annihilator, that is $[L,I]=0$. Note
that the ideal $I$ is the minimal ideal with respect to the
property that the quotient algebra $L/I$ is a Lie algebra.

\begin{defn} Let $L$ be a Leibniz algebra, $M$ be a $\mathbb{K}$ vector space and bilinear maps $[-,-]:L\times M\rightarrow M$ and $[-,-]:M\times L\rightarrow M$ satisfy the following three axioms:
\begin{equation}\label{eq_module} \begin{array}{c}[m,[x,y]]=[[m,x],y]-[[m,y],x],\\[1mm]
[x,[m,y]]=[[x,m],y]-[[x,y],m],\\[1mm]
[x,[y,m]]=[[x,y],m]-[[x,m],y]. \end{array}\end{equation}

Then $M$ is called a \emph{representation} of the Leibniz algebra $L$ or an $L-$\emph{bimodule}.
\end{defn}

\begin{defn} A $\mathbb{Z}_2$-graded vector space $L=L_0\oplus L_1$  is called a Leibniz superalgebra if it is equipped with a product $[-,-]$ which satisfies the following conditions:

\[ \big[x,[y,z]\big]=\big[[x,y],z\big] - (-1)^{\alpha \beta}\big[[x,z],y\big]- \mbox{Leibniz superidentity} \]  for all $x\in L, y\in L_{\alpha}, z\in L_{\beta}.$
\end{defn}

The vector spaces $L_0$ and $L_1$ are said to
be the even and odd parts of the superalgebra $L$,
respectively. It is obvious that $L_0$ is a Leibniz algebra and $L_1$ is a representation of $L_0.$
Note that if in Leibniz superalgebra $L$ the identity
$$[x,y]=-(-1)^{\alpha\beta} [y,x]$$ holds for any $x \in
L_{\alpha}$ and $y \in L_{\beta},$ then the Leibniz superidentity
can be transformed into the Jacobi superidentity. Thus, Leibniz
superalgebras are a generalization of Lie superalgebras and
Leibniz algebras.


\begin{defn} The set $$\mathcal{R}(L)=\left\{ z\in L\ |\ [L,
z]=0\right\}$$ is called the right annihilator  of a superalgebra
$L.$
\end{defn}

Using the Leibniz superidentity it is easy to see that
$\mathcal{R}(L)$ is an ideal of the superalgebra $L$. Moreover,
the elements of the form $[a,b]+(-1)^{\alpha \beta}[b,a],$ ($a \in
L_{\alpha}, \ b \in L_{\beta}$) belong to $\mathcal{R}(L)$.

\begin{defn}  An $L$-bimodule is called \emph{simple} or \emph{irreducible}, if it does not admit non-trivial $L$-subbimodules. An $L$-bimodule
 is called \emph{indecomposable}, if it is not a direct sum of its $L$-subbimodules.
An $L$-bimodule M is called \emph{completely reducible} if for any $L$-subbimodule $N$ there exists a complementing
$L$-subbimodule $N'$ such that $M = N \oplus N'.$
\end{defn}


In [3] it is proved that a finite-dimensional simple $L$-bimodule is either symmetric or antisymmetric for any finite-dimensional Leibniz algebra $L$.

\begin{lem} \cite{Barnes} \label{LemBar} Let $L$ be a finite-dimensional Leibniz algebra, and let $M$ be a finite-dimensional
simple $L$-bimodule. Then either $[L,M] =
0$ or $[x,m] = -[m,x]$ for all $x \in L$ and $m \in M.$
\end{lem}

An $L$-bimodule with trivial left actions is called \textit{symmetric}. If the left action is the negative of the right action, then it is called \textit{antisymmetric}.

In particular, in case of $L$ is isomorphic to three dimensional simple Lie algebra $\mathfrak{sl}_2$ we have that there exist a basis $\{x_0, x_1, \dots, x_n\}$ of $M$ such that one of the table of multiplication is hold:
\begin{equation}\label{eq_N_1} N_1:\left\{\begin{array}{lll}
[x_i,h]=(n-2i)x_i,& [h,x_i]=-(n-2i)x_i\\[1mm]
[x_i,f]=x_{i+1},& [f,x_i]=-x_{i+1},\\[1mm]
[x_i,e]=-i(n-i+1)x_{i-1},& [e,x_i]=i(n-i+1)x_{i-1}\\[1mm]
\end{array}\right.\end{equation}

\begin{equation} \label{eq_N_2} N_2:\left\{\begin{array}{lll}
[x_i,h]=(n-2i)x_i,& [h,x_i]=0,\\[1mm]
[x_i,f]=x_{i+1},& [f,x_i]=0,\\[1mm]
[x_i,e]=-i(n-i+1)x_{i-1},& [e,x_i]=0.
\end{array}\right.\end{equation}

%

In \cite{Kurb} it is studied indecomposable complex finite-dimensional Leibniz algebra bimodule over $\mathfrak{sl}_2$ that as a Lie algebra module is split into a direct sum of two simple $\mathfrak{sl}_2$-modules and it is proved that in this case there are only two indecomposable Leibniz $\mathfrak{sl}_2$-bimodules.

 \begin{thm} \cite{Kurb}\label{thmkurb}. An $\mathfrak{sl}_2$-module $M=X\oplus Y$, where $X$ and $Y$ are simple $\mathfrak{sl}_2$-modules is indecomposable as a Leibniz $\mathfrak{sl}_2$-bimodule if and only if $dim X-dim Y=2.$ Moreover, upto $\mathfrak{sl}_2$-bimodule isomorphism there are only two indecomposable $\mathfrak{sl}_2$-bimodules, which in basis $\{x_0,x_1,\dots,x_n,y_0,\dots,y_{n-2}\}$ have the following brackets:
\[ M_1:\begin{array}{lll}
[x_i,h]=(n-2i)x_i,& [h,x_i]=-(n-2i)x_i-2iy_{i-1},\\[1mm]
[x_i,f]=x_{i+1},& [f,x_i]=-x_{i+1}+y_{i},\\[1mm]
[x_i,e]=-i(n-i+1)x_{i-1},& [e,x_i]=i(n-i+1)x_{i-1}+i(i-1)y_{i-2},\\[1mm]
[y_j,h]=(n-2-2j)y_j,& [h,y_j]=0,\\[1mm]
[y_j,f]=y_{j+1},& [f,y_j]=0,\\[1mm]
[y_j,e]=-j(n-j-1)y_{j-1},& [e,y_j]=0.\\[1mm]
\end{array}\]
\[M_2:\begin{array}{lll}
[x_i,h]=(n-2i)x_i,& [h,x_i]=0,\\[1mm]
[x_i,f]=x_{i+1},& [f,x_i]=0,\\[1mm]
[x_i,e]=-i(n-i+1)x_{i-1},& [e,x_i]=0,\\[1mm]
[y_j,h]=(n-2-2j)y_j,& [h,y_j]=2(n-j-1)x_{j+1}-(n-2j-2)y_{j},\\[1mm]
[y_j,f]=y_{j+1},& [f,y_j]=x_{j+2}-y_{i+1},\\[1mm]
[y_j,e]=-j(n-j-1)y_{j-1},& [e,y_j]=(n-j-1)((n-j)x_{j}+jy_{j-1}).\\[1mm]
\end{array}\]
\end{thm}

Following theorem generalize of Theorem \ref{thmkurb} in case of $M$ is a direct sum of $k$ simple $\mathfrak{sl}_2$-modules.

\begin{thm}\cite{Kurb2}\label{thmkurb2} Let $M$ be an $\mathfrak{sl}_2$-bimodule and as a right $\mathfrak{sl}_2$-module let it decompose as $M = V_1\oplus V_2 \oplus \dots \oplus V_n,$ where $V_k$ are simple $\mathfrak{sl}_2$-modules with base $\{v_0^i, \dots, v_0^i\},$ $1 \leq k \leq n$
and $dimV_1 \geq dimV_2 \geq \dots \geq dimV_n.$ Then $M$ is an indecomposable Leibniz $\mathfrak{sl}_2$-bimodule only if $dimV_{k} - dimV_{k+1} = 2$ for all $1 \leq k \leq n - 1.$ Moreover, up to $\mathfrak{sl}_2$-bimodule isomorphism there are exactly two indecomposable
$\mathfrak{sl}_2$-bimodules:
\small{\[ M_3:\begin{array}{lll}
[v_i^{2p-1},h]=(n-4p+4-2i)v_i^{2p-1},& [h,v_i^{2p-1}]=0,\\[1mm]
[v_i^{2p-1},f]=v_{i+1}^{2p-1},& [f,v_i^{2p-1}]=0,\\[1mm]
[v_i^{2p-1},e]=-i(n-4p+5-i)v_{i-1}^{2p-1},& [e,v_i^{2p-1}]=0,\\[1mm]
[v_i^{2p},h]=(n-4p+2-2i)v_i^{2p},& [h,v_i^{2p}]=2(n-2p-i+3)v_{i+1}^{2p-1}-(n-2p-2i+2)v_{i+1}^{2p}-2iv_{i-1}^{2p+1},\\[1mm]
[v_i^{2p},f]=v_{i+1}^{2p},& [f,v_i^{2p-1}]=v_{i+2}^{2p-1}-v_{i+1}^{2p}+v_{i}^{2p+1},\\[1mm]
[v_i^{2p},e]=-i(n-4p+1-i)v_{i-1}^{2p},& [e,v_i^{2p}]=(n-2p-i+3)((n-2p-i+4)v_{i}^{2p-1}+iv_{i-1}^{2p})+i(i-1)v_{i-2}^{2p+1},\\[1mm]
[v_i^{2p+1},h]=(n-4p+1-i)v_{i}^{2p+1},& [h,v_i^{2p+1}]=0,\\[1mm]
[v_i^{2p+1},f]=v_{i+1}^{2p+1},& [f,v_i^{2p+1}]=0,\\[1mm]
[v_i^{2p+1},e]=-i(n-4p-1-i)v_{i-1}^{2p+1},& [e,v_i^{2p+1}]=0.\\[1mm]
\end{array}\]}
for all $1\leq p\leq \frac{k}{2},$

\small{\[M_4:\begin{array}{lll}
[v_i^{1},h]=(n-2i)v_i^{1},& [h,v_i^{1}]=-(n-2i)v_i^1-2iv_{i-1}^2,\\[1mm]
[v_i^{1},f]=v_{i+1}^{1},& [f,v_i^{1}]=-v_{i+1}^1+v_i^2,\\[1mm]
[v_i^{1},e]=-i(n-i+1)v_{i-1}^{1},& [e,v_i^{1}]=i(n-i+1)v_{i-1}^1+i(i-1)v_{i-2}^2,\\[1mm]
[v_i^{2p},h]=(n-4p+2-2i)v_i^{2p},& [h,v_i^{2p}]=0,\\[1mm]
[v_i^{2p},f]=v_{i+1}^{2p},& [f,v_i^{2p-1}]=0,\\[1mm]
[v_i^{2p},e]=-i(n-4p+1-i)v_{i-1}^{2p},& [e,v_i^{2p}]=0,\\[1mm]
[v_i^{2p+1},h]=(n-4p-2i)v_{i}^{2p+1},& [h,v_i^{2p+1}]=(n-4p-i+1)v_{i+1}^{2p}-(n-4p-2i)v_{i+1}^{2p+1}-2iv_{i-1}^{2p+2},\\[1mm]
[v_i^{2p+1},f]=v_{i+1}^{2p+1},& [f,v_i^{2p+1}]=v_{i+2}^{2p}-v_{i+1}^{2p+1}+v_{i}^{2p+2},\\[1mm]
[v_i^{2p+1},e]=-i(n-4p-1-i)v_{i-1}^{2p+1},& [e,v_i^{2p+1}]=(n-4p-i+1)((n-4p-i+2)v_{i}^{2p}+iv_{i-1}^{2p+1})+i(i-1)v_{i-2}^{2p+2}.\\[1mm]
\end{array}\]}
for all $1\leq p\leq \frac{k-1}{2}$, where $n=dimV_1.$
\end{thm}

\section{Main part}

In this section we describe Leibniz superalgebra which even part is $\mathfrak{sl}_2$.

\begin{lem}\label{lemm3.1}  Let $L = L_0\oplus L_1$ be a Leibniz superalgebra, such that $L_0$ is a semi-simple Lie algebra. Then $[x,y]=[y,x]$ for any $x,y \in L_1.$
\end{lem}

\begin{proof} Note that for any $x,y \in L_1$  an element $[x,y]-[y,x]$ belongs to the right annihilator of $L.$ Consequantly,
$[x,y]-[y,x]$ belongs to the center of the semi-simple Lie algebra $L_0.$ Since the center of semi-simple Lie algebra is zero we have that
$[x,y]=[y,x].$ \end{proof}

\begin{lem} \label{lem lie} Let $L = L_0\oplus L_1$ be a Leibniz superalgebra, such that $L_0$ is a semi-simple Lie algebra and $[L_0,L_1]=0$, then  $[L_1,L_1]=0.$
\end{lem}

\begin{proof} By the condition of the Lemma we have that $[x,y]=0$  for any $x\in L_0$ and $y\in L_1.$
Considering the Leibniz superidentity for the elements $x\in L_0$ and $y, z \in L_1$ we have
$$[x,[y,z]]=[[x,y],z]+[[x,z],y] = 0,$$
which follows that $[y,z]\in \mathcal{R}(L_0).$  Since $\mathcal{R}(L_0)=0,$ we derive that $[y,z]=0$ for any $y, z \in L_1.$
\end{proof}

Let $L=\mathfrak{sl}_2\oplus L_1$ be a Leibniz superalgebra, i.e., even part is isomorphic to $\mathfrak{sl}_2.$ If $L_1$ is a simple $\mathfrak{sl}_2$-bimodule, then by Lemma \ref{LemBar} we have that either $[L_1, \mathfrak{sl}_2] =0$ or $[x,y]=-[y,x]$ for all $x\in \mathfrak{sl}_2,$ $y\in L_1.$

In case of $[L_1, \mathfrak{sl}_2] =0$, from Lemma \ref{lem lie}, we derive that $[L_1, L_1]=0$ and $L$ is isomorphic to the Leibniz algebra with the following muliplication \cite{O-R-T}:
\[\left\{\begin{array}{lll}
[e,h]=2e, & [h,f]=2f, & [e,f]=h,\\[1mm]
[h,e]=-2e, & [f,h]=-2f,& [f,e]=-h,\\[1mm]
[x_i,h]=(n-2i)x_i,&
[x_i,f]=x_{i+1},&
[x_i,e]=-i(n-i+1)x_{i-1},\ 0\leq i\leq n.
\end{array}\right.\]

Therefore, it is sufficient to consider the case when $[x,y]=-[y,x]$ for all $x\in \mathfrak{sl}_2,$ $y\in L_1.$ In the following Proposition we describe such Leibniz superalgebras in case of $dim L_1=2.$

\begin{prop} \label{lem2} Let $L=\mathfrak{sl}_2 \oplus L_1$ be a Leibniz superalgebra, such that $L_1$ is a simple bimodule. Let $ dim L_1 = 2 $ and $[x,y]=-[y,x]$ for all $x\in \mathfrak{sl}_2,$ $y\in L_1,$ then $L$ is isomorphic one of the following two Leibniz superalgebras:
\[S_1:\left\{\begin{array}{llllll}
[e,h]=2e, & [h,f]=2f, &[e,f]=h,\\[1mm]
[h,e]=-2e, & [f,h]=-2f,& [f,e]=-h,\\[1mm]
[x_0,h]=x_0,& [x_1,h]=-x_1,&[x_0,f]=x_{1},&[x_1,e]=-x_{0}, \\[1mm]
[h,x_0]=-x_0, & [h,x_1]=x_1,&[f,x_0]=-x_{1},& [e,x_1]=x_{0}.\\[1mm]
\end{array}\right.\]

\[S_2:\left\{\begin{array}{llll}
[e,h]=2e, & [h,f]=2f, & [e,f]=h,\\[1mm]
[h,e]=-2e, & [f,h]=-2f,& [f,e]=-h,\\[1mm]
[x_0,h]=x_0, & [x_1,h]=-x_1,&[x_0,f]=x_{1}, &[x_1,e]=-x_{0},\\[1mm]
[h,x_0]=-x_0, & [h,x_1]=x_1,&[f,x_0]=-x_{1},&[e,x_1]=x_{0},\\[1mm]
[x_0,x_0]=2e,& [x_1,x_1]=2f,&[x_0,x_1]=h,& [x_1,x_0]=h.\\[1mm]
\end{array}\right.\]
\end{prop}

\begin{proof}
From Lemma \ref{LemBar} and \eqref{eq_N_1} we have the following products:
\[\left\{\begin{array}{llll}
[[e,h]=2e, & [h,f]=2f, & [e,f]=h,\\[1mm]
[h,e]=-2e, & [f,h]=-2f,& [f,e]=-h,\\[1mm]
[x_0,h]=x_0, & [x_1,h]=-x_1,&[x_0,f]=x_{1}, &[x_1,e]=-x_{0},\\[1mm]
[h,x_0]=-x_0, & [h,x_1]=x_1,&[f,x_0]=-x_{1},&[e,x_1]=x_{0}.\\[1mm]
\end{array}\right.\]

Put
\[\begin{cases}
[x_0,x_0]=a_{0,0}e+b_{0,0}f+c_{0,0}h,\\[1mm]
[x_0,x_1]=a_{0,1}e+b_{0,1}f+c_{0,1}h,\\[1mm]
[x_1,x_1]=a_{1,1}e+b_{1,1}f+c_{1,1}h.\\[1mm]
\end{cases}\]

Consider following Leibniz superidentities:

\begin{itemize}
  \item $[e,[x_0,x_0]]=2[[e,x_0],x_0]=0.$

On the other hand: $[e,[x_0,x_0]]=[e,a_{0,0}e+b_{0,0}f+c_{0,0}h]=b_{0,0}h+2c_{0,0}e,$ which implies
\[b_{0,0}=c_{0,0}=0.\]

  \item $[e,[x_1,x_1]]=2[[e,x_1],x_1]=2[x_0,x_1]=2a_{0,1}e+2b_{0,1}f+2c_{0,1}h,$

On the other hand: $[e,[x_1,x_1]]=[e,a_{1,1}e+b_{1,1}f+c_{1,1}h]=b_{1,1}h+2c_{1,1}e,$ which derive
\[b_{0,1}=0, \ c_{1,1}=a_{0,1}, \ b_{1,1}=2c_{0,1}.\]

  \item $[h,[x_1,x_1]]=2[[h,x_1],x_1]=2[x_1,x_1]=2a_{1,1}e+2b_{1,1}f+2c_{1,1}h,$

On the other hand: $[h,[x_1,x_1]]=[h,a_{1,1}e+b_{1,1}f+c_{1,1}h]=-2a_{1,1}e+2b_{1,1}f,$ which implies
\[a_{1,1}=0, \ c_{1,1}=0\]

  \item $[f,[x_0,x_0]]=2[[f,x_0],x_0]=-2[x_1,x_0]=-2c_{0,1}h.$

On the other hand: $[f,[x_0,x_0]]=[f,a_{0,0}e]=-a_{0,0}h,$ which derive
\[2c_{0,1}=a_{0,0}.\]
\end{itemize}

The rest Leibniz superidentities give us the same restrictions. Thus, we have remaining multiplications:
\[\begin{cases}
[x_0,x_0]=2 c_{0,1}e,& [x_1,x_1]=2 c_{0,1}f,\\[1mm]
[x_0,x_1]=c_{0,1}h,& [x_1,x_0]=c_{0,1}h.\\[1mm]
\end{cases}\]

In case of $c_{0,1} = 0$ we have superalgebra $S_1$ and if $c_{0,1} \neq 0$, then taking the change
\[x_0'=\frac{1}{\sqrt{c_{0,1}}}x_0,\ x_1'=\frac{1}{\sqrt{c_{0,1}}}x_1,\]
we obtain the Leibniz superalgebra $S_2$.
\end{proof}

It should be Remark that in the superalgebra $S_1$ the multiplication $[L_1, L_1]$ is zero, thus $S_1$ is a Leibniz algebra.

In the following Proposition we investigate the case when $ dim L_1 \geq  3.$

\begin{prop} \label{prop3.4} Let $L=\mathfrak{sl}_2 \oplus L_1$ be a Leibniz superalgebra, such that $L_1$ is a simple bimodule.
Let $ dim L_1 \geq 3 $ and $[x,y]=-[y,x]$ for all $x\in \mathfrak{sl}_2,$ $y\in L_1,$ then $[L_1, L_1]=0$.
\end{prop}

\begin{proof}
By the condition of Proposition we have that there exist a basis $\{e,f,h,x_0,_1,\dots,x_n\}$ of $L$ such that
\[\left\{\begin{array}{lll}
[e,h]=2e, & [h,f]=2f, & [e,f]=h,\\[1mm]
[h,e]=-2e, & [f,h]=-2f,& [f,e]=-h,\\[1mm]
[x_k,h]=(n-2k)x_k,&[h,x_k]=(2k-n)x_k,\\[1mm]
[x_k,f]=x_{k+1},&[f,x_k]=-x_{k+1},\\[1mm]
[x_k,e]=-k(n+1-k)x_{k-1},&[e,x_k]=k(n+1-k)x_{k-1}.\\[1mm]
\end{array}\right.\]

Put
\[[x_i,x_j]=a_{i,j}e+b_{i,j}f+c_{i,j}h, \quad 0\leq i\leq j \leq n.\]

Now we consider Leibniz superidentity for the elements $[h,[x_i,x_j]],$ $[f,[x_i,x_j]]$ and $[e,[x_i,x_j]].$
$$[h,[x_i,x_j]]=[[h,x_i],x_j]+[[h,x_j],x_i]=(2i-n)[x_i,x_j]+(2j-n)[x_j,x_i]=$$
$$=2(i+j-n)(a_{i,j}e+b_{i,j}f+c_{i,j}h)=2(i+j-n)a_{i,j}e+2(i+j-n)b_{i,j}f+2(i+j-n)c_{i,j}h,$$
where $0 \leq i \leq j \leq n.$

$$[f,[x_i,x_j]]=[[f,x_i],x_j]+[[f,x_j],x_i]=-[x_{i+1},x_j]-[x_{j+1},x_i]=$$
$$=-(a_{i+1,j}+a_{i,j+1})e-(b_{i+1,j}+b_{i,j+1})f-(c_{i+1,j}+c_{i,j+1})h,$$
where $0 \leq i \leq j \leq n-1.$

$$[e,[x_i,x_j]]=[[e,x_i],x_j]+[[e,x_j],x_i]=i(n+1-i)[x_{i-1},x_j]+j(n+1-j)[x_{j-1},x_i]=$$
$$=(i(n+1-i)a_{i-1,j}+j(n+1-j)a_{i,j-1})e+(i(n+1-i)b_{i-1,j}+j(n+1-j)b_{i,j-1})f+$$
$$+(i(n+1-i)c_{i-1,j}+j(n+1-j)c_{i,j-1})h,$$
where $1 \leq i \leq j \leq n.$

On the other hand,
$$[h,[x_i,x_j]]=[h,a_{i,j}e+b_{i,j}f+c_{i,j}h]=-2a_{i,j}e+2b_{i,j}f,$$
$$[f,[x_i,x_j]]=[f,a_{i,j}e+b_{i,j}f+c_{i,j}h]=-a_{i,j}h-2c_{i,j}f,$$
$$[e,[x_i,x_j]]=[e,a_{i,j}e+b_{i,j}f+c_{i,j}h]=b_{i,j}h+2c_{i,j}e,$$

Comparing the coefficients at the basis elements, we obtain following restrictions:
\begin{equation}\label{eq_3.1}\begin{cases}
(i+j+1-n)a_{i,j}=0,\\[1mm]
(i+j-1-n)b_{i,j}=0,\\[1mm]
(i+j-n)c_{i,j}=0.\\[1mm]
\end{cases}\end{equation}

\begin{equation} \label{eq_3.2} \begin{cases}
a_{i+1,j}+a_{j+1,i}=0,\\[1mm]
b_{i+1,j}+b_{j+1,i}=2c_{i,j},\\[1mm]
c_{i+1,j}+c_{j+1,i}=a_{i,j}.\\[1mm]
\end{cases}\end{equation}

\begin{equation} \label{eq_3.3}\begin{cases}
i(n+1-i)a_{i-1,j}+j(n+1-j)a_{i,j-1}=2c_{i,j},\\[1mm]
i(n+1-i)b_{i-1,j}+j(n+1-j)b_{i,j-1}=0,\\[1mm]
i(n+1-i)c_{i-1,j}+j(n+1-j)c_{i,j-1}=b_{i,j}.\\[1mm]
\end{cases}\end{equation}

It is obvious that from \eqref{eq_3.1} we have:
\begin{equation} \label{eq_3.4}\begin{cases}
a_{i,j}=0, & i+j\neq n-1,\\[1mm]
b_{i,j}=0, & i+j\neq n+1,\\[1mm]
c_{i,j}=0, & i+j\neq n.\\[1mm]
\end{cases}\end{equation}

From the first equality of \eqref{eq_3.2} we derive:
\begin{equation}\label{eq_3.5} a_{i,n-i-1}=(-1)^ia_{0,n-1},\quad 0 \leq i\leq n-1.\end{equation}

If $n$ is even, then we have $a_{n-1,0} = -a_{0,n-1}.$ On the other hand, $a_{n-1,0} = a_{0,n-1}.$ Thus, we get
$a_{0,n-1}=0,$ which implies $a_{i,j}=0,\ 0\leq i\leq j\leq n.$

If $n$ is odd, then in case of $n=3$ from the Leibniz superidentity $[x_1,[x_1,x_1]]-2[[x_1,x_1],x_1]=0$ we get $\ a_{1,1}=0.$
If $n\geq 4,$ then considering Leibniz superidentity $[x_{n-2},[x_1,x_1]]-2[[x_{n-2},x_1],x_1]=0,$ we obtain $\ a_{n-2,0}=0.$ According to the equality \eqref{eq_3.5} it follows $a_{i,j}=0, \ 0\leq i\leq j\leq n.$

Since $a_{i,j}=0,$ then from the equality \eqref{eq_3.3} we have:
\[b_{i,j}=0, \quad c_{i,j}=0, \quad 0\leq i\leq j\leq n.\]

Therefore, we get $ [L_1, L_1] = 0.$
\end{proof}

Now we consider the Leibniz superalgebra $L=\mathfrak{sl}_2\oplus L_1$, such that $L_1$ is a split module into direct sum of the two simple $\mathfrak{sl}_2$-modules.
Then by the Theorem \ref{thmkurb} we obtain that $L_1$ as a $\mathfrak{sl}_2$-bimodule isomorphic to $M_1$ or $M_2.$

\begin{prop}\label{lem M1} Let $L=\mathfrak{sl}_2\oplus L_1$ ba a Leibniz superalgebra such that $L_1$ is a bimodule isomorphic to $M_1$. Then
$[L_1,L_1]=0.$
\end{prop}
\begin{proof} By the condition of the Proposition we have that there exist a basis $\{e,f,h, x_0, x_1,\dots, x_n, y_0, y_1,\dots, y_{n-2}\}$
of the superalgebra $L$ such that the following products are hold:
\[\left\{\begin{array}{lll}
[e,h]=2e, & [h,f]=2f, &[e,f]=h,\\[1mm]
[h,e]=-2e, & [f,h]=-2f,& [f,e]=-h,\\[1mm]
[x_i,h]=(n-2i)x_i,& [h,x_i]=-(n-2i)x_i-2iy_{i-1},\\[1mm]
[x_i,f]=x_{i+1},& [f,x_i]=-x_{i+1}+y_{i},\\[1mm]
[x_i,e]=-i(n-i+1)x_{i-1},& [e,x_i]=i(n-i+1)x_{i-1}+i(i-1)y_{i-2},\\[1mm]
[y_j,h]=(n-2-2j)y_j,& [h,y_j]=0,\\[1mm]
[y_j,f]=y_{j+1},& [f,y_j]=0,\\[1mm]
[y_j,e]=-j(n-j-1)y_{j-1},& [e,y_j]=0.\\[1mm]
\end{array}\right.\]

Since $$[y_0,h]+[h,y_0]=(n-2)y_{0}, \qquad [y_j,f]+[f,y_j]=y_{j+1}, \quad 0 \leq j \leq n-3,$$ we have that $y_j \in \mathcal{R}(L)$ for $0 \leq j \leq n-2.$ Thus, we have $[y_i, y_j] = [x_i, y_j]=0$ for any values $i$ and $j.$

Put
\[[x_i,x_j]=a_{i,j}e+b_{i,j}f+c_{i,j}h, \quad 0\leq i\leq j \leq n.\]

Considering Leibniz superidentities for the for the elements $[h,[x_i,x_j]],$ $[f,[x_i,x_j]]$ and $[e,[x_i,x_j]],$ we obtain following restrictions:
\begin{equation}\label{eq_3.6}\begin{cases}
(i+j+1-n)a_{i,j}=0,\\[1mm]
(i+j-1-n)b_{i,j}=0,\\[1mm]
(i+j-n)c_{i,j}=0.\\[1mm]
\end{cases}\end{equation}

 \begin{equation}\label{eq_3.7}\begin{cases}
a_{i+1,j}+a_{j+1,i}=0,\\[1mm]
b_{i+1,j}+b_{j+1,i}=2c_{i,j},\\[1mm]
c_{i+1,j}+c_{j+1,i}=a_{i,j}.\\[1mm]
\end{cases} \end{equation}

\begin{equation}\label{eq_3.8}\begin{cases}
i(n+1-i)a_{i-1,j}+j(n+1-j)a_{i,j-1}=2c_{i,j},\\[1mm]
i(n+1-i)b_{i-1,j}+j(n+1-j)b_{i,j-1}=0,\\[1mm]
i(n+1-i)c_{i-1,j}+j(n+1-j)c_{i,j-1}=b_{i,j}.\\[1mm]
\end{cases}\end{equation}

Analogously, to the proof of Proposition \ref{prop3.4} from \eqref{eq_3.6} -- \eqref{eq_3.8} we conclude that $a_{i,j}=b_{i,j}=c_{i,j}=0.$
Therefore, $[x_i, x_j] = 0$ for any $0 \leq i \leq j \leq n,$ which implies $[L_1, L_1]=0.$

\end{proof}

\begin{prop} Let $L=\mathfrak{sl}_2\oplus L_1$ ba a Leibniz superalgebra such that $L_1$ is a bimodule isomorphic to $M_2$. Then
$[L_1,L_1]=0.$
\end{prop}
\begin{proof}

By the condition of the Proposition we have that there exist a basis $\{e,f,h, x_0, x_1,\dots, x_n, y_0, y_1,\dots, y_{n-2}\}$
of the superalgebra $L$ such that the following products are hold:
\[M_2:\left\{\begin{array}{lll}
[e,h]=2e, & [h,f]=2f, &[e,f]=h,\\[1mm]
[h,e]=-2e, & [f,h]=-2f,& [f,e]=-h,\\[1mm]
[x_i,h]=(n-2i)x_i,& [h,x_i]=0,\\[1mm]
[x_i,f]=x_{i+1},& [f,x_i]=0,\\[1mm]
[x_i,e]=-i(n-i+1)x_{i-1},& [e,x_i]=0,\\[1mm]
[y_j,h]=(n-2-2j)y_j,& [h,y_j]=2(n-j-1)x_{j+1}-(n-2j-2)y_{j},\\[1mm]
[y_j,f]=y_{j+1},& [f,y_j]=x_{j+2}-y_{j+1},\\[1mm]
[y_j,e]=-j(n-j-1)y_{j-1},& [e,y_j]=(n-j-1)((n-j)x_{j}+jy_{j-1}).\\[1mm]
\end{array}\right.\]

Since $$[x_0,h]+[h,x_0]=n x_{0}, \qquad [x_i,f]+[f,x_i]=x_{i+1}, \quad 0 \leq i \leq n-1,$$
we have that $x_i \in \mathcal{R}(L)$ for $0 \leq i \leq n.$ Thus, we have $[x_j, x_i] = [y_j, x_i]=0$ for any values $i$ and $j.$

Put
\[[y_i,y_j]=a_{i,j}e+b_{i,j}f+c_{i,j}h, \quad 0\leq i\leq j \leq n-2.\]

Analogously,  to the proof of Proposition \ref{prop3.4} considering Leibniz superidentities for the elements $[h,[y_i,y_j]],$ $[f,[y_i,y_j]]$ and $[e,[y_i,y_j]],$
we obtain following restrictions:

\begin{equation}\label{eq_3.9}\begin{cases}
(i+j+3-n)a_{i,j}=0,\\[1mm]
(i+j+1-n)b_{i,j}=0,\\[1mm]
(i+j+2-n)c_{i,j}=0.\\[1mm]
\end{cases}\end{equation}

\begin{equation}\label{eq_3.10}\begin{cases}
a_{i+1,j}+a_{j+1,i}=0,\\[1mm]
b_{i+1,j}+b_{j+1,i}=2c_{i,j},\\[1mm]
c_{i+1,j}+c_{j+1,i}=a_{i,j}.\\[1mm]
\end{cases}\end{equation}

\begin{equation}\label{eq_3.11}\begin{cases}
i(n-1-i)a_{i-1,j}+j(n-1-j)a_{i,j-1}=2c_{i,j},\\[1mm]
i(n-1-i)b_{i-1,j}+j(n-1-j)b_{i,j-1}=0,\\[1mm]
i(n-1-i)c_{i-1,j}+j(n-1-j)c_{i,j-1}=b_{i,j}.\\[1mm]
\end{cases}\end{equation}

From the equations \eqref{eq_3.9} -- \eqref{eq_3.11}, similarly of the proof of
Proposition \ref{prop3.4} we conclude that $ a_{i,j} = 0, \ b_{i,j} = 0, \ c_{i,j} = 0. $
%
Therefore, we obtain $[L_1,L_1]=0.$
\end{proof}

In the next two Propositions we consider Leibniz superalgebra $L=\mathfrak{sl}_2\oplus L_1$, such that $L_1$ is a split $\mathfrak{sl}_2-$modules as $V_1\oplus V_2\oplus\dots\oplus V_k$. By Theorem \ref{thmkurb2} we obtain that $L_1$ is a bimodule isomorphic to $M_3$ or $M_4$.

\begin{prop} Let $L=\mathfrak{sl}_2\oplus L_1$ ba a Leibniz superalgebra such that $L_1$ is a bimodule isomorphic to $M_3$. Then
$[L_1,L_1]=0.$
\end{prop}

\begin{proof}
By the condition of the Proposition we have that there exist a basis $\{e,f,h,v_0^j,v_1^j,\dots,v_{n-2k+2}^j\}, \ 1\leq j \leq k$ of the superalgebra $L$ such that the following products are holds:

\[\left\{\begin{array}{lll}
[e,h]=2e,&[h,f]=2f,\ [e,f]=h,\\[1mm]
[h,e]=-2e,&[f,h]=-2f,\ [f,e]=-h,\\[1mm]
[v_i^{2p-1},h]=(n-4p+4-2i)v_i^{2p-1},& [h,v_i^{2p-1}]=0,\\[1mm]
[v_i^{2p-1},f]=v_{i+1}^{2p-1},& [f,v_i^{2p-1}]=0,\\[1mm]
[v_i^{2p-1},e]=-i(n-4p+5-i)v_{i-1}^{2p-1},& [e,v_i^{2p-1}]=0,\\[1mm]

[v_i^{2p},h]=(n-4p+2-2i)v_i^{2p},& [h,v_i^{2p}]=2(n-2p-i+3)v_{i+1}^{2p-1}-(n-2p-2i+2)v_{i+1}^{2p}-2iv_{i-1}^{2p+1},\\[1mm]
[v_i^{2p},f]=v_{i+1}^{2p},& [f,v_i^{2p-1}]=v_{i+2}^{2p-1}-v_{i+1}^{2p}+v_{i}^{2p+1},\\[1mm]
[v_i^{2p},e]=-i(n-4p+1-i)v_{i-1}^{2p},& [e,v_i^{2p}]=(n-2p-i+3)((n-2p-i+4)v_{i}^{2p-1}+iv_{i-1}^{2p})+i(i-1)v_{i-2}^{2p+1},\\[1mm]

[v_i^{2p+1},h]=(n-4p+1-i)v_{i}^{2p+1},& [h,v_i^{2p+1}]=0,\\[1mm]
[v_i^{2p+1},f]=v_{i+1}^{2p+1},& [f,v_i^{2p+1}]=0,\\[1mm]
[v_i^{2p+1},e]=-i(n-4p-1-i)v_{i-1}^{2p+1},& [e,v_i^{2p+1}]=0.\\[1mm]
\end{array}\right.\]

Since
\[\begin{array}{lll}
[v_i^{2p-1},e]+[e,v_i^{2p-1}]=-i(n-4p+5-i)v_{i-1}^{2p-1},\\[1mm]
[v_i^{2p-1},f]+[f,v_i^{2p-1}]=v_{i+1}^{2p-1},\\[1mm]
[v_i^{2p-1},h]+[h,v_i^{2p-1}]=-i(n-4p+5-i)v_{i-1}^{2p-1},
\end{array}\]
we have that $v_i^{2p-1} \in\mathcal{R}(L),$ $1\leq p \leq \big[\frac{k+1}{2}\big],$  $0\leq i\leq  n-4p+4.$
Then according to Lemma \ref{lemm3.1} we have $$[v_i^{2p-1},v_j^{q}]=[v_j^{q},v_i^{2p-1}]=0, \quad 1\leq q\leq  k, \ 0\leq j\leq  n-2p+2.$$

 Put \[[v_i^{2p},v_j^{2q}]=a_{i,j}^{p,q}e+b_{i,j}^{p,q}f+c_{i,j}^{p,q}h,\quad 1\leq p\leq q\leq k, \ 0\leq i\leq n-4p+2, \ 0\leq j\leq n-4q+2.\]

Considering Leibniz superidentities for the elements $[v_0^{2p-1},[v_i^{2p},v_j^{2q}]]$ and $[v_1^{2p-1},[v_i^{2p},v_j^{2q}]]$ we have $[V_{2p},V_{2q}]=0.$ Which implies $[L_1,L_1]=0.$

\end{proof}

\begin{prop} Let $L=\mathfrak{sl}_2\oplus L_1$ ba a Leibniz superalgebra such that $L_1$ is a bimodule isomorphic to $M_4$. Then
$[L_1,L_1]=0.$
\end{prop}
\begin{proof}
The proof of this Proposition is proved similarly to the previous one.
\end{proof}



\end{document}